\documentclass[a4paper, reqno]{amsart}
\usepackage[dvipsnames]{xcolor}
\usepackage{amsmath,amsthm,amssymb,relsize,mathrsfs,csquotes,microtype}
\usepackage[margin=30mm]{geometry}

\usepackage{enumitem}
\setlist{nosep,font=\normalfont,beginpenalty=10000, midpenalty=10000}

\definecolor{darkblue}{rgb}{0,0,0.6}

\usepackage[breaklinks, pdftex, ocgcolorlinks,colorlinks=true, citecolor=darkblue, filecolor=darkblue, linkcolor=darkblue, urlcolor=darkblue]{hyperref}

\makeatletter
\newtheorem*{rep@theorem}{\rep@title}
\newcommand{\newreptheorem}[2]{%
	\newenvironment{rep#1}[1]{%
		\def\rep@title{#2 \ref{##1}}%
		\begin{rep@theorem}}%
		{\end{rep@theorem}}}
\makeatother

\newtheorem{proposition}{Proposition}[section]
\newtheorem{theorem}[proposition]{Theorem}
\newtheorem{corollary}[proposition]{Corollary}
\newtheorem{lemma}[proposition]{Lemma}
\newreptheorem{theorem}{Theorem}
\newreptheorem{proposition}{Proposition}

\theoremstyle{definition}
\newtheorem{definition}[proposition]{Definition}
\newtheorem{question}[proposition]{Question}

\theoremstyle{remark}
\newtheorem{remark}[proposition]{Remark}

\newtheorem*{remark*}{Remark}

\newcommand{\Z}{\mathbb{Z}}

\newcommand{\N}{\mathbb{N}}

\newcommand{\R}{\mathbb{R}}

\DeclareMathOperator{\h}{H}

\def\op{\operatorname}

\newcommand{\defeq}{\mathrel{\mathop:}=}

\newcommand{\ttimes}{\mathbin{%
    \ooalign{\raise1.15ex\hbox{$\scriptstyle\sim$}\cr\hidewidth$\times$\hidewidth\cr}%
    }}

\renewcommand{\~}{\nobreak\ }

\begin{document}

\title{The unsolvability of the homeomorphism problem}

\author{Stefan Friedl}
\author{Tobias Hirsch}
\author{Marc Kegel}

\address{Fakult\"at f\"ur Mathematik\\Universit\"at Regensburg\\93040 Regensburg, Germany}
\email{sfriedl@gmail.com}
\email{tobias.hirsch@ur.de}

\address{Universidad de Sevilla, Dpto.\ de Álgebra,
Avda.\ Reina Mercedes s/n,
41012 Sevilla, Spain}
\email{kegelmarc87@gmail.com}

\def\subjclassname{\textup{2020} Mathematics Subject Classification}
\expandafter\let\csname subjclassname@1991\endcsname=\subjclassname


\date{\today} 

\begin{abstract}
In this short expository note, we present detailed proofs of three fundamental unsolvability results in topology: Markov's theorem on the unsolvability of the homeomorphism problem, the existence of unrecognizable manifolds in all dimensions larger than $3$, and Novikov's theorem on the unrecognizability of the $n$-sphere for $n\geq5$.
\end{abstract}

\maketitle

\section{Introduction}
\addtocontents{toc}{\protect\setcounter{tocdepth}{1}}

As a corollary of the resolution of the geometrization conjecture\~\cite{perelman1,perelman2} the homeomorphism and diffeomorphism problem is solved for closed orientable manifolds of dimension less than or equal to\~$3$, see for example\~\cite{Kuperberg}. 
On the contrary, a well-known theorem due to Markov\~\cite{Markov} from 1958 states that the homeomorphism problem for manifolds of dimension larger than $3$ is not solvable. For different write-ups of this result we refer to\~\cite{Stanko, Chernavsky,Kirby_Markovsthm,Gordon,Tancer,BooneHakenPoenaru} and\~\cite[Exercises\~5.1.10(c),\~5.2.2(c)]{GS}. The main idea in all these proofs and related results is to produce a solution to the undecidable group triviality problem from such an algorithm.

In this expository article, we first make these undecidability statements precise and then present detailed proofs of them. The input data of an algorithm is supposed to be a finite set. In our setting, we will work with PL-simplicial manifolds, as defined in Definition\~\ref{def:combinatorial-manifold}. (Note that these objects go by different names, at times they are also called combinatorial manifolds or stellar manifolds, see e.g.\ \cite{Glaser}.) 
Given a PL-simplicial manifold $M$, we can also consider the corresponding topological realization $|M|$ which comes with a natural PL-structure. Note that two PL-simplicial manifolds are combinatorially homeomorphic (i.e.\ stellar equivalent) if and only if the PL-manifolds given by their topological realizations are PL-homeomorphic.

With this background on PL-manifolds we can precisely state the unsolvability theorems.

\begin{theorem}
\label{thm:main_PL_input_and_output}
For all integers $n\geq4$ there exists no algorithm that 
\begin{itemize}
    \item takes as input connected closed orientable $n$-dimensional PL-simplicial manifolds  $M_1$ and $M_2$, and 
    \item outputs whether or not $|M_1|$ is PL-homeomorphic to $|M_2|$.
\end{itemize}
\end{theorem}

Theorem\~\ref{thm:main_PL_input_and_output} is satisfying because the in- and output are both from the category of PL-simplicial manifolds. 
The theorem is an immediate consequence of the following more general theorem,  with in- and outputs mixing different equivalence relations, namely the following five (in general different) equivalence relations on $n$-dimensional PL-manifolds:
\begin{enumerate}
    \item being PL-homeomorphic,
    \item being homeomorphic,
    \item being homotopy equivalent,
    \item admitting isomorphisms between all homotopy groups,
    \item admitting an isomorphism between the fundamental groups.
\end{enumerate}

\begin{theorem}
\label{thm:main_triangulation_input_and_arbitrary_output}
For each of the equivalence relations $(1)$--$(5)$ and for all integers $n\geq4$, there exists no algorithm that 
\begin{itemize}
    \item takes as input connected closed orientable $n$-dimensional PL-simplicial manifolds $M_1$ and $M_2$, and 
    \item outputs whether or not $|M_1|$ is equivalent to $|M_2|$.
\end{itemize}
\end{theorem}

The set of closed orientable $n$-dimensional manifolds up to homeomorphism is countable\~\cite{Cheeger_Kister}. Similarly, it is a folk result that the set of closed orientable $n$-dimensional smooth manifolds up to diffeomorphism is countable. Thus, topological and smooth manifolds can, in principle, be used as input data for algorithms. On the other hand, we do not know of any good way to present such manifolds naturally by data that could be used as input for algorithms. Possible solutions for 4-dimensional topological manifolds are discussed in\~\cite{Freedman_Zuddas,algorithm_4_manifold}. An approach to represent smooth manifolds is discussed in\~\cite{BooneHakenPoenaru}, although it is less practical than the other approaches based on simplicial complexes. In dimensions $n\leq6$ the PL-homeomorphism class of a PL-simplicial manifold $M$ determines a smooth structure (unique up to diffeomorphism) on $|M|$, see for example \cite[Theorem 2]{Milnor_46later}.\footnote{In general there exist PL-simplicial manifolds that admit no smooth structure or several non-diffeomorphic smooth structures, see for example\~\cite{Milnor_46later}.} And thus Theorem\~\ref{thm:main_PL_input_and_output} also makes a statement about the smooth category in dimensions $n=4,5,6$. 

In dimension $4$, there exists another natural presentation method for compact smooth manifolds. Indeed any such manifold can be presented by a Kirby diagram\~\cite{GS}. From an algorithmic viewpoint, Kirby diagrams of compact 4-dimensional smooth  manifolds are equivalent to triangulations of compact 4-dimensional smooth manifolds, see  \cite[Lemma 8.2]{algorithm_4_manifold} for details. Thus Theorem\~\ref{thm:main_PL_input_and_output} implies the following corollary.

\begin{corollary}
\label{cor:main_Kirby_input_and_output_diff}
There exists no algorithm that 
\begin{itemize}
    \item takes as input Kirby diagrams\footnote{We can see a Kirby diagram as a link labelled with integers (for $2$-handles) or dots (for $1$-handles). There are several known ways to represent diagrams of links (up to planar isotopy), for example via PD codes~\cite{knotatlas,Mast}, DT codes~\cite{dt_codes}, Gau\ss\ codes~\cite{KnotInfo}, braid words~\cite{Artin}, or isosignatures~\cite{regina}.} of connected closed orientable 4-dimensional smooth manifolds $M_1$ and $M_2$, and 
    \item outputs whether or not $M_1$ is diffeomorphic to $M_2$.\qed
\end{itemize}
\end{corollary}

\subsection*{Unrecognizable manifolds}

In fact, the proof of Theorem\~\ref{thm:main_triangulation_input_and_arbitrary_output} shows something slightly stronger, namely that for no $n\geq4$ does there exist an algorithm that takes as input a connected closed orientable $n$-dimensional PL-simplicial manifold $M$ (or a Kirby diagram for $n=4$) and outputs whether or not there exists an integer $l\geq0$ such that $|M|$ is equivalent to $\#_l(S^2\times S^{n-2})$.

Using a slightly stronger formulation of the underlying undecidability result from group theory, we prove in Section\~\ref{sec:unrecognizable} along the same lines a more general statement, saying informally that there are fixed manifolds which are unrecognizable. The precise statement is the following.

\begin{theorem}
\label{thm:unrecognizable_different_categories}
There exists an integer $l\geq0$, such that for each of the equivalence relations $(1)$--$(5)$ and for all integers $n\geq4$, there exists no algorithm that 
\begin{itemize}
    \item takes as input a connected closed orientable $n$-dimensional PL-simplicial manifold $M$
    , and 
    \item outputs whether or not $|M|$ is equivalent to $\#_l (S^2 \times S^{n-2})$.
\end{itemize}
\end{theorem}

In the same way as Corollary\~\ref{cor:main_Kirby_input_and_output_diff} is implied by Theorem\~\ref{thm:main_triangulation_input_and_arbitrary_output}, Theorem\~\ref{thm:unrecognizable_different_categories} implies the following. 

\begin{corollary}
\label{cor:unrecognizable_Kirby}
There exists an integer $l\geq0$ such that there exists no algorithm that 
\begin{itemize}
    \item takes as input a Kirby diagram of a connected closed orientable 4-dimensional smooth manifold $M$, and 
    \item outputs whether or not $M$ is diffeomorphic to $\#_l(S^2 \times S^2)$.\qed
\end{itemize}
\end{corollary}

In higher dimensions it was shown by Novikov\~\cite{Novikov_sphere} with an observation by Chernavsky, Leksine\~\cite{Chernavsky} that no manifold is recognizable.

\begin{theorem}\label{thm:novikov}
For each of the equivalence relations $(1)$--$(5)$, all integers $n\geq5$ and every connected closed orientable $n$-dimensional PL-simplicial manifold $X$, there exists no algorithm that
\begin{itemize}
    \item takes as input a connected closed orientable $n$-dimensional PL-simplicial manifold $M$, and
    \item outputs whether or not $|M|$ is equivalent to $|X|$.
\end{itemize}
\end{theorem}

\begin{remark}\label{rem:recognizing-PL-manifolds} 
For $n\geq6$, it is undecidable whether or not an abstract simplicial complex is an $n$-dimensional PL-simplicial manifold in the sense of Definition\~\ref{def:combinatorial-manifold}. Indeed, let $M$ be a connected closed orientable $(n-1)$-dimensional PL-simplicial manifold and consider its suspension $\Sigma M$. The links of the two suspension vertices are both equal to $M$, and hence $\Sigma M$ is a PL-simplicial $n$-manifold if and only if $M$ is PL-homeomorphic to $S^{n-1}$. Thus, an algorithm recognizing PL-simplicial $n$-manifolds would give an algorithm recognizing $S^{n-1}$, contradicting Theorem~\ref{thm:novikov}. In contrast, for $n\leq4$ it is decidable whether a finite abstract simplicial complex is an $n$-dimensional PL-simplicial manifold: one simply checks whether all vertex links are PL-homeomorphic to $S^{n-1}$, and sphere recognition is decidable in dimensions at most $3$~\cite{Rubinstein,Thompson}. 
\end{remark}

In dimension $4$, recognition of the $4$-sphere remains open; see for example~\cite{Weinberger,Kirby,Gordon,Tancer} or Problem~4.53 in~\cite{K3}. Consequently, an algorithm recognizing $5$-dimensional PL-simplicial manifolds would in particular give an algorithm recognizing the $4$-sphere. The decidability of $5$-dimensional PL-simplicial manifold recognition is likewise open.

\begin{question}
For each of the equivalence relations $(1)$--$(5)$, what is the minimal integer $l\geq0$ such that there exists no algorithm that 
\begin{itemize}
    \item takes as input a connected closed orientable 4-dimensional PL-simplicial manifold $M$, and 
    \item outputs whether or not $|M|$ is equivalent to $\#_l (S^2 \times S^2)$?
\end{itemize}
\end{question}

In recent years, several improvements to the known minima have been made. 
The strongest result to date is obtained by Tancer\~\cite{Tancer}, who proved that $\#_\ell(S^2\times S^2)$ is unrecognizable with regard to relations $(2)$--$(5)$ for any $\ell\geq9$. His proof does not apply to (1) as he uses Freedman's classification of closed simply-connected $4$-manifolds up to homeomorphism~\cite{Freedman}. The best result for PL-homeomorphism is that $\#_\ell(S^2\times S^2)$ is PL-unrecognizable for all $\ell\geq12$ by work of Gordon\~\cite{Gordon}, see also\~\cite{Chernavsky}.\footnote{Note added (November 2026). Kegel, Li and Ren~\cite{KLR} proved that \(\#_\ell(S^2\times S^2)\) is topologically unrecognizable for every \(\ell\geq 7\) and smoothly unrecognizable for every \(\ell\geq 9\). Using their group, Gordon~\cite{Gordon_recognition} lowered the PL bound to \(7\). More recently, \cite{Ren} lowered both bounds to \(2\). Related recent work appeared in~\cite{DKLR,gordon_knotted_spheres}.}

\begin{remark} \label{rem:RE}
Since the group triviality problem is $\mathrm{RE}$-hard, where $\mathrm{RE}$ denotes the class of recursively enumerable problems, our reductions in fact show that the corresponding topological problems are $\mathrm{RE}$-hard as well. We do not attempt to determine which of these problems are $\mathrm{RE}$-complete. As our aim is to keep the exposition accessible to topologists without a background in computability theory, we will restrict ourselves in the following to the more basic decidability and undecidability statements. 
\end{remark}

\subsection*{Conventions}
With a manifold we always mean a topological manifold. If the manifold has extra attributes such as a PL-structure or a smooth structure, we mention this explicitly.

We take a pragmatic approach to the notion of algorithms. (The underlying precise notion is in terms of Turing machines~\cite{boolos_burgess_jeffrey} or equivalent models of computation.) For each algorithm, we will explain how the input is represented as finite data. We will then describe the algorithms in standard mathematical language, giving enough details that the translation to a concrete algorithmic setting is evident.

\subsection*{Acknowledgments}
The idea for this article started during the workshop on \textit{Algorithms in $4$-manifold topology}, hosted at Universit\"at Regensburg in September 2024, funded by the SFB\~1085 \emph{Higher Invariants} (Universit\"at Regensburg, funded by the DFG, ID 224262486), and organised by Stefan Friedl, Marc Kegel, and Birgit Tiefenbach. We thank the participants of the workshop for useful discussions. Special thanks go to Martin Tancer for giving a beautiful introductory talk on his work \cite{Tancer} and to Lisa Schambeck and Matthias Uschold for initial discussions about this subject. MK thanks Nathan Dunfield, Shana Li, and Qiuyu Ren for interesting and useful discussions about these topics. Moreover, we thank Cameron Gordon for useful comments and the anonymous referees for careful reading and interesting feedback. In particular, Remarks~\ref{rem:recognizing-PL-manifolds} and~\ref{rem:RE} and the exposition about Novikov's theorem in Section~\ref{sec:unrecognizable} are based on suggestions of the referees.

The article was written during a week-long visit of SF and TH at the IMUS (Instituto de matemáticas, Universidad de Sevilla) in February 2026. We thank the VII Plan Propio de Investigación y Transferencia of the University of Sevilla for funding of the visit and the IMUS for providing office space.

\subsection*{Statement on AI use}
No generative AI tools were used for this paper.

\subsection*{Individual grant support}
MK is supported by a Ram\'on y Cajal grant \mbox{(RYC2023-043251-I)} and PID2024-157173NB-I00 funded by MCIN/AEI/10.13039/501100011033, by ESF+, and by FEDER, EU; and by a VII Plan Propio de Investigación y Transferencia (SOL2025-36103) of the University of Sevilla.

\section{The homeomorphism problem}
In this section, we will assume that the reader has some familiarity with concepts from combinatorial topology and PL-topology. We refer to \cite{zeeman,hudson,Glaser,rourkesandersonnew,friedlmonster} for details.  

\begin{definition}\label{def:combinatorial-manifold}\hfill
\begin{enumerate}
\item Two finite abstract simplicial complexes $K$ and $L$ are \emph{combinatorially homeomorphic} if $K$ and $L$ are stellar equivalent, i.e.\ if they are related by  a sequence of stellar subdivisions, stellar welds and simplicial isomorphisms. 
\item Let $n\in \N_0$. 
 A finite abstract simplicial complex is a 
\emph{combinatorial $n$-ball} if it is combinatorially homeomorphic to the abstract simplicial complex \mbox{$\Delta^n\defeq\{\{0,\dots,n\},\mathcal{P}(\{0,\dots,n\})\}$} and a 
\emph{combinatorial $n$-sphere} if it is combinatorially homeomorphic to the abstract simplicial complex $\partial \Delta^n\defeq\{\{0,\dots,n\},\mathcal{P}(\{0,\dots,n\})\setminus \{0,\dots,n\}\}$. 
\item An abstract simplicial complex $K$ is an \emph{$n$-dimensional \mbox{PL-simplicial} manifold} if for every vertex $v$ the link $\op{Lk}(K,v)$ is a combinatorial $(n-1)$-sphere or a combinatorial $(n-1)$-ball.
\item The boundary $\partial K$ of an $n$-di\-men\-sional PL-simplicial manifold $K$ is the union of all simplices $s$ such that  the link $\op{Lk}(K,s)$ is a combinatorial ball. By \cite[Corollary\~II.4]{Glaser} we know that $\partial K$ is again a PL-simplicial manifold.
We say $K$ is \emph{closed} if $K$ is finite and $\partial K=\emptyset$.
\item An \emph{$n$-dimensional PL-manifold} (note the missing \enquote{simplicial}) is a topological space together with a homeomorphism to the topological realization of an $n$-dimensional PL-simplicial manifold.
\end{enumerate}
\end{definition}

In \cite[Chapter 6]{rourkesandersonnew} a theory of PL-handle decompositions is developed. When adapted to our setting, it essentially parallels the well-known theory of handle decompositions for smooth manifolds, i.e.\ a $k$-handle on a PL-simplicial manifold $W$ is a $\Delta^k\times\Delta^{n-k}$ attached along a \mbox{PL-embedding} $\partial\Delta^k\times\Delta^{n-k}\to\partial W$. For technical reasons we choose the PL-structure on $\Delta^k\times\Delta^{n-k}$ that is obtained by taking the product PL-structure after barycentric subdivision of $\Delta^k$. This will be needed to ensure that the $k$-handle collapses onto its core. Also note that \cite{rourkesandersonnew} deals with PL-manifolds PL-embedded in some $\R^n$. While we in principle work with abstract manifolds, it will later be necessary to embed them. The treatment in \cite{rourkesandersonnew} then ensures that the necessary handle slides and cancellations are still possible.

Using handle attachments we associate PL-simplicial manifolds to a finite group presentation.  

\begin{definition}\label{dfn:relalizeP}
    Let  $P=\langle g_1,\dots, g_k\mid r_1,\dots, r_l\rangle$ be a finite group presentation and $n\geq4$. A handle decomposition on an $(n+1)$-dimensional PL-simplicial manifold $W$ \emph{realizes $P$} if it is obtained in the following manner:
    \begin{itemize}
        \item First build a PL-simplicial manifold $W_1$ by attaching $k$ many PL 1-handles via orientation-preserving attaching maps to a PL-ball $D^{n+1}$. 

        Note that up to reordering there exists a preferred isomorphism between $\pi_1(W_1)$ and $\langle g_1,\dots, g_k\rangle$ and that the inclusion $\partial W_1\to W_1$ induces an isomorphism on fundamental groups as we only attach handles of codimension at least 3. 
        \item Now attach $l$ many PL 2-handles to $W_1$ via attaching maps $\phi_1,\dots,\phi_l\colon\partial\Delta^2\times \Delta^{n-1}\to\partial W_1$ where $\phi_i(\partial \Delta^2\times\{0\})$ represents $r_i\in\langle g_1,\dots,g_k\rangle$ under the above isomorphism to form $W_2$
        \item Finally attach $k$ many additional 2-handles along nullhomotopic curves to $W_2$. 
    \end{itemize}
\end{definition}

To apply this construction in our argument, it is important that we can find such a manifold realizing a given group presentation algorithmically. We begin by recalling the following terminology.

\begin{definition}
    A set $X$ is \emph{recursively enumerable} if there exists an algorithm which, when run indefinitely, outputs the elements of $X$, allowing repetitions and allowing no output if $X=\emptyset$.
\end{definition}

It is part of this definition that the elements of $X$ are of a form that can be the output of an algorithm. Observe the following algorithmic lemma using a technique known as \enquote{dovetailing}, compare \cite[Corollary 7.15]{boolos_burgess_jeffrey}.

\begin{lemma}\label{lem:recenum}
    Let $X$ be a recursively enumerable set and $\mathcal A$ be an algorithm that takes as input an element $x\in X$. The subset $Y\subseteq X$ of elements on which $\mathcal A$ halts is recursively enumerable. 
\end{lemma}

\begin{proof}
    The algorithm $\mathcal A$ consists of a sequence of computational steps. Construct a recursive enumeration of $Y$ as follows:
    \begin{itemize}
        \item Run step 1 of $\mathcal A$ on the 1\textsuperscript{st} output of the recursive enumeration of $X$.
        \item Run step 1 of $\mathcal A$ on the 2\textsuperscript{nd} output of the recursive enumeration of $X$.
        \item Run step 2 of $\mathcal A$ on the 1\textsuperscript{st} output of the recursive enumeration of $X$.
        \item Run step 1 of $\mathcal A$ on the 3\textsuperscript{rd} output of the recursive enumeration of $X$.
        \item Run step 2 of $\mathcal A$ on the 2\textsuperscript{nd} output of the recursive enumeration of $X$.
        \item Run step 3 of $\mathcal A$ on the 1\textsuperscript{st} output of the recursive enumeration of $X$.
        \item $\dots$
    \end{itemize}
    If this step-wise execution of $\mathcal A$ halts on an output of the recursive enumeration of $X$, output it. 
\end{proof}

In the next lemma we algorithmically find a manifold realizing $P$ that embeds into $\R^{n+1}$. However, algorithmically constructing PL-embeddings into $\R^{n+1}$ is technically involved. We therefore do not provide a direct algorithm, but first enumerate all subcomplexes of $\R^{n+1}$ and then use the last lemma to filter out those corresponding to PL-simplicial manifolds with the desired properties.

\begin{lemma}\label{lem:algrelP}
    For every $n\geq4$, there exists an algorithm that
    \begin{itemize}
    \item takes as input a  finite group presentation $P$, and
    \item outputs an $(n+1)$-dimensional PL-simplicial manifold $W$ which PL-embeds into $\mathbb R^{n+1}$ and allows a handle decomposition realizing $P$.
\end{itemize} 
\end{lemma}

\begin{proof}
    A PL-simplicial manifold $W$ PL-embeds into $\R^{n+1}$ if and only if there exists a PL-simplicial manifold PL-homeomorphic to $S^{n+1}$ with $W$ as a subcomplex (see \mbox{\cite[Theorem\~2.14]{rourkesandersonnew}}). The set of compact PL-simplicial manifolds (up to simplicial isomorphism) that are PL-hom\-e\-o\-mor\-phic to $S^{n+1}$ is recursively enumerable. This is obtained by starting with some $S^{n+1}$ and systematically applying all (finitely many) possible stellar subdivisions and stellar welds, then iterating this process on the results. As each PL-structure on $S^{n+1}$ is finite, each such PL-simplicial manifold has finitely many subcomplexes. It follows that the set $X$ of subcomplexes (up to simplicial isomorphism) of PL-simplicial manifolds PL-homeomorphic to $S^{n+1}$ is recursively enumerable.   
    
    We now apply Lemma\~\ref{lem:recenum} with the following algorithms:
    \begin{itemize}
        \item There exists an algorithm with input a finite abstract simplicial complex $W$ that halts if and only if $W$ is a compact $(n+1)$-dimensional PL-simplicial manifold. This again follows by systematically applying stellar subdivisions and stellar welds to the links of the finitely many vertices of $W$ to check if they are combinatorial $n$-spheres or combinatorial $n$-balls. If this is the case, the process will eventually stop.
        \item There exists an algorithm with input a compact PL-simplicial manifold $W$ that halts if and only if $W$ allows a handle decomposition realizing $P$. This is the case as there are only finitely many possibilities for the location of the handles. 
    \end{itemize}
    Thus the set of $(n+1)$-dimensional PL-simplicial manifolds (up to simplicial isomorphism) that PL-embed into $\R^{n+1}$ and allow a handle decomposition realizing $P$ is recursively enumerable. This provides in particular an algorithm for constructing such manifolds if they exist. 
    
    Thus, it remains to show that this set is non-empty. But in doing so we no longer need to perform constructions algorithmically:
    We realize the presentation complex $X$ of $P$ as a 2-dimensional finite abstract simplicial complex. It therefore PL-embeds into $\R^{n+1}$ for $n\geq2\cdot2$ by a general position argument \cite[Theorem\~5.4]{rourkesandersonnew}. Hence, there exists a PL-simplicial structure on $\R^{n+1}$ such that $X$ arises as a subcomplex (see \cite[Theorem\~2.14]{rourkesandersonnew}). A regular neighbourhood of $X$ in it has a handle decomposition realizing $P$ by \cite[Proposition\~6.9]{rourkesandersonnew}.
\end{proof}

We will crucially use that $W$ embeds into $\R^{n+1}$ because this restricts how the 2-handles can be attached. In a certain sense, it forces all of them to be \enquote{untwisted}. We use this in the following lemma which is the main idea in Markov's proof\~\cite{Markov} and says that the PL-homeo\-mor\-phism type of a manifold realizing $P$ detects whether $P$ presents the trivial group. 

\begin{lemma}\label{lem:PtrivialiffS2}
    Let $n\geq4$ and $P$ be a finite group presentation with $l$ relations. Let $W$ be an $(n+1)$-dimensional PL-simplicial manifold that allows a PL-handle decomposition realizing $P$ and PL-embeds into $\R^{n+1}$. Then $P$ presents the trivial group if and only if $\partial W$ is PL-homeomorphic to $\#_l(S^2\times S^{n-2})$.
\end{lemma}

\begin{proof}
    For the \enquote{if}-direction observe that $P$ presents the fundamental group of $W$ and the inclusion $\partial W\to W$ induces an isomorphism on fundamental groups as $W$ has a handle decomposition where all handles have codimension at least 3.

    Now assume that $P$ presents the trivial group. We first prove that $W$ has a PL-handle decomposition consisting of a 0-handle and $l$ many 2-handles. As above, the inclusion $\partial W_2\to W_2$ induces an isomorphism on fundamental groups. By construction, $\pi_1(W_2)$ is presented by $P$, hence $\pi_1(\partial W_2)$ is trivial. This implies that the attaching circles of the $2$-handles attached in the last bullet point of Definition~\ref {dfn:relalizeP} are homotopic in $\partial W_2$ to loops that cancel the $1$-handles as in \cite[Lemma 6.4]{rourkesandersonnew}. Since the dimension of $\partial W_2$ is at least\~4, they are also PL-isotopic to such a loop (compare \cite[Corollary 5.9]{rourkesandersonnew}). It follows that $W$ has a handle decomposition given by attaching the $l$ many 2-handles from the second bullet point of Definition~\ref {dfn:relalizeP} to a $D^{n+1}$.

    Let $A$ be the wedge of 2-spheres formed by the cores $D^2\times \{0\}$ of the 2-handles $D^2\times D^{n-1}$ and the cone of their attaching circles in the 0-handle $D^{n+1}$. After replacing the combinatorial structure on $D^{n+1}$ by the cone of $\partial D^{n+1}$ in the given structure, we may assume that $A\subseteq W$ is a subcomplex. This is still a PL 0-handle as a sequence of stellar subdivisions or stellar welds on $\partial D^{n+1}$ induces one on its cone, hence the PL-homeomorphism type of $W$ remains unchanged. By \cite[Corollary\~3.30]{rourkesandersonnew} $W$ is a regular neighbourhood of $A$ in $\R^{n+1}$ since $W\hookrightarrow\R^{n+1}$ has \mbox{codimension-0}. As a PL-embedding $\bigvee_l S^2\hookrightarrow \R^{n+1}$ is unique up to PL-isotopy by \cite[Theorem\~1]{Husch}, the uniqueness of regular neighbourhoods \cite[Theorem\~3.8]{rourkesandersonnew} yields that $W$ is unique up to PL-homeomorphism. Since the $l$-fold boundary connected sum of $S^2\times D^{n-1}$ is certainly a possibility, the claim follows. 
\end{proof}

This gives us the topological tools to reduce Theorem\~\ref{thm:main_triangulation_input_and_arbitrary_output} to the unsolvability of the triviality problem for finitely presented groups, which we now recall precisely.

\begin{definition}
    An \emph{Adian--Rabin set} is a recursively enumerable set $S$ of finite group presentations such that there exists no algorithm that 
    \begin{itemize}
        \item takes as input a presentation $P$ from $S$, and
        \item outputs whether or not $P$ presents the trivial group.
    \end{itemize}
\end{definition}

\begin{theorem}[Novikov\~\cite{Novikov-algo-insolv}]\label{thm:triviality_problem}
    There exists an Adian--Rabin set. In particular, the set of all finite group presentations is an Adian--Rabin set.
\end{theorem}

\begin{proof}[Proof of Theorem\~\ref{thm:main_triangulation_input_and_arbitrary_output}]
We consider an integer $n\geq 4$ and one of the following equivalence relations on $n$-dimensional PL-manifolds:
\begin{enumerate}
    \item being PL-homeomorphic,
    \item being homeomorphic,
    \item being homotopy equivalent,
    \item admitting isomorphisms between all homotopy groups,
    \item admitting an isomorphism between the fundamental groups.
\end{enumerate}
We will argue by contradiction. We assume that there exists an algorithm that
\begin{itemize}
    \item takes as input $n$-dimensional PL-simplicial manifolds $M_1$ and $M_2$ 
    \item outputs whether or not $|M_1|$ is equivalent to $|M_2|$.
\end{itemize}
We claim that the existence of such an algorithm would imply the existence of an algorithm deciding if a group is trivial and thus contradict Theorem\~\ref{thm:triviality_problem}. 

Indeed, let $P$ be a finite group presentation. Using Lemma\~\ref{lem:algrelP}, we can algorithmically construct a PL-simplicial manifold $W$ which PL-embeds into $\R^{n+1}$ and allows a handle decomposition realizing $P$. Now create a PL-simplicial manifold $T$ realizing $\#_l(S^2 \times S^{n-2})$ where $l$ is the number of relations in $P$ and use the potential algorithm to decide if $\partial W$ and $T$ represent equivalent spaces. By Lemma\~\ref{lem:PtrivialiffS2} this is the case if and only if $P$ presents the trivial group.
\end{proof}

\section{Unrecognizable manifolds}\label{sec:unrecognizable}
Proving specific manifolds are unrecognizable requires more careful study of the triviality problem in group theory. We begin by introducing the following natural terminology.

\begin{definition}    
    If $\mathfrak E$ is a property of finite group presentations, an Adian--Rabin set has property $\mathfrak E$, if all its presentations or underlying groups have it. 
\end{definition}

The proof of Theorem\~\ref{thm:unrecognizable_different_categories} is obtained in the same way as the proof of Theorem\~\ref{thm:main_triangulation_input_and_arbitrary_output} from the existence of an Adian--Rabin set with 13 relations, see Gordon {\cite[Theorem 2.2]{Gordon}}. This in particular proves that $\#_l (S^2\times S^2)$ is unrecognizable for $l=13$. Gordon's proof for $l=12$ additionally shows topologically that one can add one less 2-handle. The argument by Tancer \cite{Tancer} for $l=9$, which as mentioned does not apply to PL-homeomorphism, is based on a different algebraic approach. He constructs an Adian--Rabin set with deficiency 9 and uses this to build suitable manifolds.

To prove Theorem\~\ref{thm:novikov} we will also need further algebraic input.

\begin{definition}
    A group $G$ is \emph{superperfect} if $\h_1(G;\Z)=\h_2(G;\Z)=0$.
\end{definition}

\begin{lemma}\label{lem:suppAR}
    There exists a superperfect Adian--Rabin set.
\end{lemma}

\begin{proof}
    We first show there is a perfect Adian--Rabin set. Note that there exists an algorithm which takes as input a finite group presentation and outputs whether it is perfect. This follows by computing the Smith normal form of the corresponding presentation matrix. Thus, the set of finite presentations of perfect groups is recursive, and in particular recursively enumerable. 
    Since every non-perfect group is non-trivial, an algorithm deciding triviality among perfect presentations would, together with the above algorithm, decide the triviality problem for arbitrary finite presentations. Thus, Theorem~\ref{thm:triviality_problem} implies that the set of finite presentations of perfect groups is an Adian--Rabin set.
    
    It can be turned into a superperfect Adian--Rabin set by replacing its elements with their Schur covering groups. Given a perfect group $G$, one can algorithmically construct its Schur covering group $\widetilde G$ (see \cite[p.10]{Chernavsky}) which is superperfect and fits into an exact sequence
    \[
    1\to\h_2(G)\to\widetilde G\to G\to1.
    \]
    Since this construction is algorithmic, the resulting set of presentations is recursively enumerable. Moreover, $\widetilde G$ is trivial if and only if $G$ is trivial.
\end{proof}

We now again turn this group-theoretic undecidability into a topological statement by algorithmically associating a suitable manifold to every superperfect presentation. We proceed as in the proof of Lemma\~\ref{lem:algrelP}. In particular, we will again not provide a direct construction, instead explaining how to find the desired manifold within a larger set.

\begin{lemma}\label{lem:alghomSphereP}
    For every $n\geq5$, there exists an algorithm that
    \begin{itemize}
        \item takes as input a superperfect finite group presentation $P$, and
        \item outputs an $n$-dimensional PL-simplicial manifold $W$ which is a homology sphere with fundamental group presented by $P$.
    \end{itemize}
\end{lemma}

\begin{proof}
    By a theorem of Kervaire \cite{Kervaire} there exists an $n$-dimensional PL-simplicial manifold $W$ which is a homology sphere with fundamental group presented by $P$. Thus we can proceed by recursively enumerating the set of all finite abstract simplicial complexes up to simplicial isomorphism and applying Lemma\~\ref{lem:recenum} with the following algorithms:
    \begin{itemize}
        \item There exists an algorithm with input a finite abstract simplicial complex $W$ that halts if and only if $W$ is a compact $n$-dimensional PL-simplicial manifold.
        \item There exists an algorithm with input a finite abstract simplicial complex $W$ that halts if and only if $W$ is a homology sphere.
        \item There exists an algorithm with input a finite abstract simplicial complex $W$ that halts if and only if $\pi_1(W)$ is presented by $P$. This follows by first computing some finite group presentation of $\pi_1(W)$ and then systematically applying Tietze transformations. \qedhere
    \end{itemize}
\end{proof}

\begin{proof}[Proof of Theorem\~\ref{thm:novikov}]
We consider one of the following equivalence relations (1)--(5) and again argue by contradiction, i.e.\ we assume that there exists an algorithm that
\begin{itemize}
    \item takes as input a connected closed orientable $n$-dimensional PL-simplicial manifold $M$, and
    \item outputs whether or not $|M|$ is equivalent to $|X|$.
\end{itemize}
We claim that this implies the existence of an algorithm deciding if a superperfect group is trivial and thus contradict Lemma\~\ref{lem:suppAR}.

Indeed, let $P$ be a superperfect finite group presentation. Use Lemma~\ref{lem:alghomSphereP} to algorithmically construct an $n$-dimensional PL-simplicial manifold $W$ which is a homology sphere with fundamental group presented by $P$. Then $W\# X$ is equivalent to $X$ if and only if $P$ presents the trivial group: If $P$ presents the trivial group, the resolution of the generalized PL Poincaré conjecture for $n\geq5$ \cite{Smale, Zeeman2, KS} implies that $W$ is PL-homeomorphic to $S^n$. If $W\# X$ is equivalent to $X$, we have $\pi_1(X)\cong\pi_1(X)*\pi_1(W)$ and Grushko's theorem yields that $\pi_1(W)$ is trivial.
\end{proof}

We conclude this note by highlighting where the given proof of the unrecognisability of $S^n$ for $n\geq5$ fails for $n=4$. Firstly, Kervaire \cite{Kervaire} only proves that every superperfect finitely presented group with a balanced presentation, i.e.\ one with the same number of generators and relations, can be realized as the fundamental group of a homology 4-sphere. Kervaire needs to use $n>4$ when he wants to homotope a map from $S^2\to W$ to an embedding. It is unknown whether the triviality problem is decidable for superperfect groups with a balanced presentation. Secondly, the generalized PL Poincaré conjecture for $n=4$ remains open. Since it is resolved topologically by Freedman \cite{Freedman}, overcoming the first obstacle would still prove unrecognisability with respect to relations (2)--(5).

\let\MRhref\undefined
\bibliographystyle{hamsalpha}
\bibliography{bib.bib}

\end{document}